\documentclass[10pt]{amsart}

\usepackage{amssymb,latexsym}

\numberwithin{equation}{section}

\def\PP#1{{\mathbb P}^{#1}}

\theoremstyle{plain}

\newtheorem{theorem}{Theorem}

\newtheorem{proposition}{Proposition}
\newtheorem{lemma}{Lemma}
\theoremstyle{definition}
\newtheorem{definition}{Definition}
\newtheorem{example}{Example}
\newtheorem{notation}{Notation}
\newtheorem{remark}{Remark}

\date{}

\begin{document}

\title[]{On the $X$-rank with respect to linear projections of projective varieties}
\author{Edoardo Ballico and Alessandra Bernardi}
\address{Dept. of Mathematics\\ University of Trento\\38123 Povo (TN), Italy}
\address{CIRM -FBK \\ 38123 Povo (TN), Italy}
\email{ballico@science.unitn.it, bernardi@fbk.eu}
\thanks{The authors were partially supported by CIRM - FBK (TN - Italy), MIUR and GNSAGA of INdAM (Italy).}


\begin{abstract}
In this paper we improve the known bound for the $X$-rank $R_{X}(P)$ of an element  $P\in {\mathbb{P}}^N$ in the case in which $X\subset {\mathbb P}^n$ is a projective variety obtained as a linear projection from a general $v$-dimensional subspace $V\subset {\mathbb P}^{n+v}$. Then, if $X\subset {\mathbb P}^n$ is a curve obtained from a projection of a rational normal curve $C\subset {\mathbb P}^{n+1}$ from a point  $O\subset {\mathbb P}^{n+1}$,  we are able to describe the precise value of the $X$-rank for those points $P\in {\mathbb P}^n$ such that $R_{X}(P)\leq R_{C}(O)-1$ and to improve the general result. Moreover we give a stratification, via the $X$-rank, of the osculating spaces to projective cuspidal projective curves $X$. Finally we give a description and a new bound of the $X$-rank of subspaces both in the general case and with respect to integral non-degenerate projective curves.
\end{abstract}

\maketitle

\section*{Introduction}

The subject of this paper is the so called ``$X$-rank'' with respect to an integral, projective, non-degenerate variety $X\subset \PP n$, where $\PP n$ is the $n$-dimensional projective space defined over an algebraically closed field $K$ of characteristic $0$. 

\begin{definition}\label{Xrank}
Let $P\in \mathbb {P}^n$, the $X$-rank of $P$ is the minimum integer $R_X(P)$ for which there exist $R_X$ distinct points $P_1, \ldots , P_{R_X}\in X$ such that $P\in \langle P_1, \ldots , P_{R_X}\rangle$.
\end{definition}

The notion of $X$-rank  arises naturally in applications which want to minimize the number of elements belonging to certain projective variety $X\subset \PP n$ needed to give, with a linear combination of them, a fixed element of $\langle X\rangle=\PP n$.  
If such a $\PP n$ is a space of tensors and the variety $X$ parameterizes tensors of certain structure, then the notion of $X$-rank has actually a physical meaning and it is also called  ``structured rank'' (cfr. \cite{cst}, \cite{accf}). 

Our paper fits into the applications in the following two cases. First, if $X$ is the projection of a variety parameterizing tensors (that is a particular case of what studied in Section \ref{projectionvar}) then the structured rank defined by $X$ can be used e.g. in graphics for face recognitions (see \cite{ga}, \cite{tp}, \cite{clf}). Second, if $X$ is a rational normal curve, as in Section \ref{curve}, it corresponds to the case of a variety that parameterizes symmetric tensors in $S^dV$ of $X$-rank equal to 1 with the dimension of the vector space $V$ equal to $2$ (in some applications, tensors may be either symmetric, e.g. \cite{co}, \cite{ca}, \cite{dh}, \cite{js} or symmetric only in some modes, e.g. \cite{be}, \cite{cr}).

From a pure geometrical point of view the notion of the $X$-rank of an element  is not the most natural one to consider. In fact the set that parameterizes points of a fixed rank is not a closed variety. 

\begin{notation}\label{sigma0}
We indicate with $\sigma^0_s(X)$ the set of points of $\mathbb{P}^n$ whose $X$-rank is at most $s$, and with $\Sigma_s^{0}(X)$ the set of points of $\mathbb{P}^n$ whose $X$-rank is actually $s$, i.e $\Sigma^{0}_{s}:=\sigma_{s}^{0}(X)\setminus \sigma_{s-1}^{0}(X)$.
\end{notation}

The set  $\Sigma^{0}_{s}\subset \langle X\rangle$, is not a closed variety.  In order to get a projective variety we have to look at its Zariski closure.

\begin{definition}\label{sigma}
The $s$-th secant variety $\sigma_s(X)\subset \mathbb{P}^n$ of $X$ is the Zariski closure of $\sigma^0_s(X)$, i.e.
$$\sigma_s(X)=\overline{\bigcup_{P_1, \ldots , P_s\in X}\langle P_1, \ldots , P_s\rangle}.$$
\end{definition}

The generic element of $\sigma_{s}(X)$ has $X$-rank equal to $s$ and we will say that if $P\in \sigma_{s}(X)$ then $P$ has ``border rank'' $s$ (we can be more precise in the following definition).

\begin{definition}\label{border}
If  $P\in \sigma_s(X)\setminus \sigma_{s-1}(X)$ we say that $P$ has $X$-border rank $s$, and we write $\underline{R_X}(P)=s$.
\end{definition}

Observe that $\underline{R_X}(P)\leq R_X(P)$.

The recent paper \cite{bl} analyzes, among the other topics, the perspective of studying the $X$-ranks of points of a given border rank with respect to a variety $X$ that is the Segre embedding of $\PP {}(A) \times Y  $ to $ \PP {}(A\otimes W)$, where $Y \subset \PP {}(W)$ and $A,W$ being vector spaces. The new tool introduced in that paper to approach that problem is  a generalized notion of rank and border rank to linear subspaces (Sections 3 and 7 of \cite{bl}).

\begin{definition}\label{rankV}
Let $V \subseteq \mathbb {P}^n$ be a non-empty linear subspace. The $X$-rank $R_X(V)$ of $V$ is the minimal cardinality of a finite set $S\subset X$ such that $V \subseteq \langle  S\rangle$.
\end{definition}

Obviously $\dim (V)+1 \le R_X(V) \le n+1$ for any $V\subset \PP {n}$ and $R_X(\mathbb {P}^n)=n+1$.

In our paper (Section \ref{subspaces}) we give a contribution to that new tool and in Propositions \ref{i0} and \ref{i3} we present  general results for the $X$-rank of subspaces with respect to any complex, projective integral and non-degenerate variety $X\subset \PP n$. In the particular case in which $X$ is a curve and the subspace is a line, it is possible to give a more precise result (Theorem \ref{i1}).

The knowledge of the $X$-rank with respect to a variety $X$ parameterizing certain kind of tensors is studied also in several recent papers (\cite{cs}, \cite{bl}, \cite{lt}) for very special varieties $X$. Among the older papers we point out the examples of smooth space curves $X$ with points of $X$-rank $3$ listed in \cite{p}.
On our knowledge the only result that is nowadays known on the $X$-rank and that holds for any complex, projective, integral and non-degenerate $m$-dimensional variety $X\subset \PP n$ is due Landsberg and Teitler (see \cite{lt}, Proposition 5.1):  for all $P\in \mathbb {P}^n$: $$R_X(P) \le n+1-m.$$

In our paper we will refine that bound in the case of $X\subset \PP n$ being a projection of a projective variety $Y\subset \PP {n+v}$ from a linear space of dimension $v$. 

\begin{definition}\label{typical}
The minimum integer $\alpha_{X}$ such that $\sigma_{\alpha_{X}}(X)=\PP n$ is called the $X$-generic rank. \end{definition}

The most general result that we present  is the following theorem.

{\textbf{Theorem \ref{e0}}:} 
Let $Y\subset \mathbb {P}^{n+v}$ be an integral and non-degenerate variety of dimension $m$. Let $X_{V}\subset \mathbb {P}^n$ be the linear projection of $Y$ from a  $(v-1)$-dimensional linear subspace $V\subset \PP {n+v}$ that does not intersect $Y$.

\quad (a) If $V$ contains a general point of $\mathbb {P}^{n+v}$, then:
$$R_X(P)\le \alpha _Y\  \hbox{ for all } P\in \mathbb {P}^n.$$

\quad (b) Assume that $V$ is general. Let $b(Y,v)$ be the minimal $s$ such that $\dim (\sigma _s(Y)) \ge n+1$. Then $\lceil (n+1)/(m+1)\rceil
\le b(Y,v) \le \alpha _Y -  \lceil
(v-m-1)/(m+1)\rceil$ and
$$R_{X_V}(P) \le b(Y,v)\  \hbox{ for all } P\in \mathbb {P}^n.$$

Section \ref{projectionvar} is entirely devoted to the proof of that theorem, and some examples are given.

In Section \ref{curve} (see Lemma \ref{l1}) we can be more precise and realize the bound of Theorem \ref{e0} as an equality for the particular case of $X\subset \PP n$ being a projection of a rational normal curve $Y\subset \PP {n+1}$(that can actually be seen as the variety parameterizing either homogeneous polynomials of degree $n+1$ in two variables, or symmetric tensors in $\PP {}(S^{n}V)$ with the dimension of the vector space $V$ equal to $2$) from a point $O \in \PP {n+1}$. 
\\
Moreover if $X\subset \PP n$ is a cuspidal curve obtained as a projection of a rational normal curve $C\subset \PP {n+1}$, it is also possible to give a stratification of the elements belonging to the osculating spaces to $X$ via the $X$-rank (see Proposition \ref{cuspidal}). 
\\
\\
\textbf{Acknowledgements.} We thank the referee for an essential suggestion which gave the present form of Theorem \ref{e0}.

\section{$X$-rank with respect to linear projections of projective varieties}\label{projectionvar}

Let $X \subset \PP n$ be an irreducible variety of dimension $m$ not contained in a hyperplane.  On our knowledge, the only known general result on a bound for the $X$-rank in the general case is due to  Landsberg and Teitler (see \cite{lt}, Proposition 5.1) who proved that 
\begin{equation}\label{land}
R_{X}(P) \le n+1-m
\end{equation}
for all $P\in \mathbb {P}^n$. In this section we restrict our attention to the case of the a variety $X_{V}\subset \PP n$ obtained from the following construction.

Let $Y\subset \PP {n+v}$ be an $m$-dimensional projective variety. Let $V\subset \PP {n+v}$ be a $v$-dimensional projective linear subspace of $\PP {n+v}$ such that $V\cap Y =\emptyset$. Consider the linear projection $\ell_{V}$ of $\PP {n+v}$ onto $\PP n$ from $V$, i.e.
\begin{equation}\label{ellV}
\ell_{V}:\PP {n+v}\setminus V\to \PP n,
\end{equation}
and define the variety $X_{V}\subset \PP n$  to be the projective variety obtained as $\ell_{V}(Y)$:
\begin{equation}\label{XV}
X_{V}:=\ell_{V}(Y)\subset \PP n.
\end{equation}

In  Theorem\label{e0} \ref{e0} stated in the Introduction and proved in Subsection \ref{S2} we show that another stronger upper bound for $R_{X_{V}}(P)$, with $P\in \mathbb {P}^n$ can be given. Here we list three examples.

\begin{example}\label{e1}
Let $Y\subset \mathbb {P}^r$ be an integral and non-degenerate curve. In  \cite{a} Remark 1.6 it is proved that  $\dim (\sigma _s(Y))=\min\{r,2s-1\}$
for all $s \ge 1$. Thus, in the set-up of part (b) of Theorem \ref{e0}, if $Y$ is a curve, we have $$b(Y,v) = \lceil (n+2)/2\rceil , \hbox{ while }\alpha _Y = \lceil (n+v+1)/2\rceil.$$
\end{example}

\begin{example}\label{e2}
Fix integers $m >0$ and $d>0$. Let $Y = Y_{m,d} \subset \mathbb {P}^{{m+d\choose m}-1}$ be the $d$-th Veronese embedding of $\mathbb {P}^m$. We have 
$$\alpha_{Y_{m,d}} \le \left\lceil \binom{m+d}{m}/(m+1)\right\rceil$$
 and equality holds, except in  few exceptional cases listed in \cite{ah}, \cite{c}, \cite{bo}. This
is a deep theorem by J. Alexander and A. Hirschowitz. They proved it in \cite{ah}. Another proof can be found in \cite{c}. See \cite{bo} for a recent reformulation of it. Here the general linear projection $\ell_{V}(Y)=X_{V}$ is isomorphic to $Y $ if
$\dim (\sigma_{2}(Y)) \le n$, i.e.  if $2m+1 \le n$.
\end{example}

\begin{example}\label{e3}
Let $Y \subset \mathbb {P}^{n+v}$ be an integral and non-degenerate subvariety of dimension $m$. Let $b:= \dim (\mbox{Sing}(Y))$ with the convention
$b=-1$ if $Y$ is smooth. Assume $m> (2n+2v+b)/3 -1$. Then $\alpha_{Y}=2$ (see \cite{z}, Theorem II.2.8). The general linear projection $\ell_{V}(Y)=X_{V}$ of $Y$ is birational to $Y$, but not isomorphic to it.
\end{example}

\subsection{The proof}\label{S2}

Let $G(v-1,n+v)$ denote the Grassmannian of all $(v-1)$-dimensional linear subspaces of $\mathbb {P}^{n+v}$. 
 
For any $V\in G(v-1,n+v)$ let $\ell _V: \mathbb {P}^{n+v}\backslash V \to \mathbb {P}^n=M$ denote the linear projection from $V$ as in (\ref{ellV}). For any $P\in M$, set $V_P:= \langle  \{P\}\cup V\rangle\subset \PP {n+v}$. If $V\cap M=\emptyset$ and $P\in M$, then $\dim (V_P) =v$ and we identify the fiber $\ell _V^{-1}(P) \subset  \mathbb {P}^{n+v}\backslash V$ with $V_P \backslash V$. Now we fix an integral and non-degenerate variety $Y \subset \mathbb {P}^{n+v}$ and let $X_{V}:=\ell_{V}(Y)$ be as in (\ref{XV}). In the statement of Theorem \ref{e0} we only need the linear projection from a general $V\in G(v-1,n+v)$. We consider only linear projections $\ell _V$ from $V\in G(v-1,n+v)$ such that $V\cap Y=\emptyset$. For these subspaces $ \ell _V\vert Y$ is a finite morphism.

\begin{lemma}\label{f1}
Fix $P\in \mathbb {P}^n $ and let $X_{V}=\ell_{V}(Y)\subset \PP n$ as above.  Then
\begin{equation}\label{eqf1}
R_X(P) \le \min _{Q\in V_P\backslash V} R_Y(Q).
\end{equation}
where $V$ and $V_{P}\subset \PP {n+v}$ are as above.
\end{lemma}

\begin{proof}
Fix $Q\in V_P\backslash V$ and $S\subset Y$ computing $R_Y(Q)$. Thus $Q\in \langle  S\rangle$ and $\sharp (S)=R_Y(Q)$.
Since $S\subset Y$ and $Y\cap V=\emptyset$, the linear projection $\ell _V$ is defined at each point
of $S$. Thus $\ell _V(S)$ is a finite subset of $\ell _V(Y)=X$ and $\sharp (\ell _V(S))\le \sharp (S)$.
Since $Q\notin V$,  $Q\in \langle  S\rangle$ and $\ell _V(Q)=P$, we have $P\in \langle  \ell _V(S)\rangle$.
Thus $R_X(P) \le \sharp (\ell _V(S))\le R_Y(Q)$ for all $Q\in V_P\backslash V$.
\end{proof}

\vspace{0.3cm}

\qquad {\emph {Proof of Theorem \ref{e0}.}}
  By definition  $\sigma_{\alpha_{Y}}(Y) = \mathbb {P}^{n+v}$. Thus there exists a non-empty open subset $U\subset \mathbb {P}^{n+v}\backslash \sigma_{\alpha_{Y}-1}(Y)$ such that and $R_Y(Q) = \alpha_{Y}$ for all $Q\in U$. Take as $V$ any element of $G(v-1,n+v)$ such that $V\cap Y =\emptyset$, $V\cap M = \emptyset$ and $V\cap U \ne \emptyset$ (the space $M=\PP n$ is  the image of $\PP {n+v}$ via $\ell_{V}$). Thus $V\cap U$ is a non-empty open subset of $V$. Thus $V_P\cap U$ is a non-empty open subset of $V_P$ for all $P\in M$. Thus $\min _{Q\in V_P\backslash V} R_Y(Q) \le \alpha_{Y}$ for all $P\in M$ by applying  Lemma \ref{f1}. This proves part (a) of the theorem.

Let us now prove the two bounds for $b(Y,v)$. Obviously, $\dim (\sigma _{s+1}(X)) \le \min \{n+v, \dim (\sigma _s(X)) +m+1\}$ for all $s\geq 1$. Hence $b(Y,v) \ge \lceil (n+1)/(m+1)\rceil$ and $\alpha _Y \ge \lceil (n+v-a)/(m+1)\rceil + b(Y,v)$. Hence $b(Y,v) \le \alpha _Y -  \lceil (n+v-a)/(m+1)\rceil$. Since $a \le n+m+1$, we get $b(Y,v) \le \alpha _Y -  \lceil (v-m-1)/(m+1)\rceil$. If $b(Y,v) \ne \lceil (n+1)/(m+1)\rceil$, then we may improve the upper bound for $b(Y,v)$. Indeed, the function $s \mapsto \dim (\sigma _s(Y))-\dim (\sigma _{s-1}(Y))$ is non-decreasing (see either \cite{a2}, the formula in part (i) of Proposition 2.1 or \cite{z}, Proposition V.1.7  which holds also for singular varieties, or \cite{f}, Theorem 1.17). If this function take a value $\le m$ for some $s\in \{2,\dots ,b_{Y,v}\}$, then $\alpha _Y \ge \lceil (n+v-a)/m\rceil + b(Y,v)$.

We can now prove part (b) of the theorem. Set $a:= \dim (\sigma _{b(Y,v)}(Y))$.  Obviously $R_Y(Q) \le s$
for every $Q\in \Sigma_s^0(X)$ (see Notation \ref{sigma0}). Fix a non-empty subset $\sigma ^{+}_s(Y)$ of $\Sigma_s^0(X)$ which is open in $\sigma _s(Y)$. For every constructible subset $E$ of $\mathbb {P}^{n+v}$ let $\dim (E)$ denote the maximal dimension of an irreducible component of the algebraic set $\overline{E}$. Since $\sigma ^{+}_s(Y)$ is dense in $\sigma_s(Y)$, we get $\dim (\sigma _s(Y)\setminus \sigma ^{+}_s(Y)) < \dim (\sigma _s(Y))$. Since $Y$ is fixed, while we choose $V$ general, then the dimensional part of Bertini's theorem gives that for general $V$ and any integer $s>0$ we have $\dim (V\cap \sigma_s(Y)) = \max \{-1, \dim (\sigma_s(Y))-n-1\}$ and $\dim (V\cap (\sigma_s(Y)\setminus \sigma^{+}_s(Y))) \le \max \{-1,\dim (\sigma_s(Y))-n-2\}$. Thus the definition of the integer $b(Y,v)$  gives $\sigma ^{+}_{b(Y,v)}(Y)\cap V \ne \emptyset$. Therefore $\sigma ^{+}_{b(Y,v)}(Y)\cap V$ is open and dense in the algebraic set $\sigma _{b(Y,v)}(Y)\cap V$, which has pure dimension  $a -n-1$ and it is irreducible if $a-n-1 \ge 1$.  Fix any $P\in \mathbb {P}^n$. Since $V_P$ has codimension $n$, every irreducible component of $\sigma _{b(Y,v)}(Y)\cap V_P$ has dimension at least $a-n$. Since $\sigma ^{+}_{b(Y,v)}(Y)$ is open in $\sigma _{b(Y,v)}(Y)$,  $\sigma ^{+}_{b(Y,v)}(Y)\cap V_P$ is open in the algebraic set $\sigma _{b(Y,v)}(Y)\cap V_P$. Since $\sigma ^{+}_{b(Y,v)}(Y)\cap V \ne \emptyset$ and $V \subset V_P$, we have $\sigma ^{+}_{b(Y,v)}(Y)\cap V_P \ne \emptyset$. Since $\dim (\sigma ^{+}_{b(Y,v)}(Y)\cap V)=a-n-1$, we get $\sigma ^{+}_{b(Y,v)}(Y)\cap (V_P\setminus V) \ne \emptyset$. Thus Lemma 3.4 gives $R_X(P) \le b(Y,v)$.
\qed
\section{The $X$-rank with respect to projections of rational normal curves}\label{curve}

 In this section we apply the results of the previous one to the particular case in which $Y=C\subset \PP {n+1}$ is a rational normal curve and the subspace  $V\subset \PP {n+1}$ is a point $O\in \mathbb {P}^{n+1}\setminus Y$.   The linear projection (\ref{ellV}) becomes:
\begin{equation}\label{lO} 
\ell _O: \mathbb {P}^{n+1}\setminus \{O\}  \to \mathbb {P}^n.\end{equation}
Each point $P\in \PP n$ corresponds to a line $L_P:= \{O\}\cup \ell _O^{-1}(P)$, and each line $L$ through $O$ intersects $\PP n$ in a unique point. Now    
\begin{equation}\label{X}
X:= \ell _O(C) \subset \PP n.
\end{equation}

\begin{remark}
Since the center of the projection $O\notin C$, the curve $X\subset \PP n$ turns out to be an integral and non-degenerate subcurve of $\PP n $ of degree $\deg (X)=n+1$ and $\ell _O\vert C \to X$ is the normalization map. 
\end{remark}
 
\begin{definition}\label{tangential} We indicate with $\tau(X)$ the tangential variety of a variety $X$. Let $T^{*}_{P}(X)$ be the Zariski closure of $\cup_{y(t),z(t)\in X\\ y(0)=z(0)=P}\lim_{t\rightarrow 0}\langle y(t),z(t)\rangle$ and then define $\tau(X)=\cup_{P\in X}T^{*}_{P(X)}$.
\end{definition}

\begin{remark} Clearly $\tau(X) \subset \sigma_{2}(X)$.
\end{remark}

\begin{definition}\label{osculating}
Let $X \subset \PP n$ be a variety, and let $P \in X$ be a smooth point. We define the $k$-th osculating space to $X$ at $P$ as the linear space generated by $(k + 1)P \cap X$ (i.e. by the $k$-th infinitesimal neighborhood  of $P$ in $X$) and we denote it by $\langle (k+1)P\rangle_X$. Hence $\langle P\rangle_X = {P}$, and $\langle 2P\rangle_X = T_P(X)$ , the projectivized tangent space to $X$ at $P$.
\end{definition}

\begin{remark} 
We observe that $\sigma_{1}(X)=X$ and also that $\sigma_{s-1}(X)\subset \sigma_s(X)$. 
\end{remark}

\begin{remark}
The curve $X$ is smooth if and only if $O\notin \sigma_{2}(C)$. If $O\in \sigma_{2}(C)$, then $X$ has arithmetic genus $1$. If $O\in \sigma_{2}(C)\backslash \tau(C)$ (where $\tau(C)$ is the tangential variety of $C$ defined in Definition \ref{tangential}), then  $X$ has an ordinary node (and in this case $R_C(O)=2$). If $O\in \tau(C)$, then $X$ has and an ordinary cusp (in this case if $R_C(O)=n+1$, see \cite{Sy}, \cite{cs}, \cite{cglm} and \cite{bgi}). 
\end{remark}

\begin{remark}\label{rem} Fix $P\in \PP n$ and let  $S\subset X\subset \PP n$ be a finite subset computing $R_X(P)$, i.e. $\sharp (S)=R_X(P)$ and $P\in \langle  S\rangle$. The set $S'\subset C \subset \PP {n+1}$ such that $\sharp (S')=\sharp (S)$ and $\ell _O(S')=S$ is uniquely determined by $S$, unless $X\subset \PP n$ is nodal and the node belongs to $S$. If the node belongs to $S$, then $S'$ is uniquely determined if the preimages of the node are prescribed.

If $O\in \sigma_{2}(C)\backslash \tau(C)$, call $Q\in X$ the singular point of $X$ and call  $Q',Q''\in C$ the points of $C$ mapped onto $Q$ by $\ell _{O}$. 
Since $P\in \langle  S\rangle$ then $L_P\cap \langle  S'\rangle \ne \emptyset$. Since $\ell _O$ is a linear projection, $\ell _O\vert S'$ is injective and $\ell _O(S')$ is linearly independent, $S'$ is linearly independent and $O\notin L_P\cap \langle  S'\rangle \ne \emptyset$. Thus $L_P\cap \langle  S'\rangle$ is a unique point $P_S\in \PP {n+1}$ and $P_S\ne O$. Conversely, if we take any linearly independent $S_1 \subset C$ (with the restriction that if $X$ is nodal, then $Q''\notin S_1$) and $O\notin \langle  S_1\rangle$, then $\ell _O\vert S_1$ is injective and $\ell _O(S_1)$ is linearly independent. Hence $S'$ computes $R_C(P_S)$ unless $R_C(P_S)$ is computed only by subsets whose linear span contains $O$ or, in the nodal case, by subsets containing $Q''$. Thus, except in these cases, $R_X(P) = R_C(P_S)$. In the latter case for a fixed $P$ and $S$ we could exchange the role of $Q'$ and $Q''$, but still we do not obtain in this way the rank.
\end{remark}

\begin{lemma}\label{l1} Let $X\subset \PP n$ be the linear projection of a rational normal curve $C\subset \PP {n+1}$ from a point $O\in \PP {n+1}\setminus C$ (as in (\ref{X})). If $P\in \mathbb {P}^n$, define $L_{P}\subset \PP {n+1}$ to be the line $L_{P}:=\langle O,P\rangle$.
Fix $P\in \mathbb {P}^n$ a point such that $R_X(P) \le R_C(O)-1$. Then
\begin{equation}\label{eqa1}
R_X(P) = \min _{A\in L_P\backslash \{O\}} R_C(A).
\end{equation}
\end{lemma}

\begin{proof}
Let $A\in L_P\backslash \{O\}$ and take $S_{A}\subset C$ computing $R_C(A)$. Now $P\in \langle \ell_{O}(S_{A})\rangle$ for all $A\in L_{P}\setminus\{O\}$, then $R_{X}(P)\leq \min_{A\in L_{P}\setminus\{O\}}R_{C}(A)$. The other inequality is done above in Remark \ref{rem}.
\end{proof}

Remark \ref{rem}, together with the knowledge of  the $C$-ranks of a rational normal curve (see \cite{bgi}, Theorem 3.8) immediately give the following result.

\begin{proposition}\label{l2} Let $X\subset \PP n$ be the linear projection of a rational normal curve $C\subset \PP {n+1}$ from a point $O\in \PP {n+1}\setminus C$ (as in (\ref{X})).
Fix $P\in \mathbb {P}^n$ and take a zero-dimensional scheme $Z \subset X$ with minimal length such that $P\in \langle  Z\rangle$. If $X$ is singular, then assume that $Z_{red}$ does not contain the singular point of $X$. Set $z:= \mbox{length}(Z)$. If $Z$ is reduced, then $R_X(P) =z$ by definition of $X$-rank. If $Z$ is not reduced, then $R_X(P) \le n+3-z$.
\end{proposition}

\begin{proposition}\label{cuspidal}
 Let $X\subset \mathbb {P}^n$, $n \ge 3$, be a non-degenerate integral curve such that $\deg (X)=n+1$ and $X$ has a cusp in $Q\in X$. Let $C\subset \PP {n+1}$ be the rational normal curve such that $\ell_O(C)=X$ for $O\in T_{Q'}(C)$ and $Q'\in C$. Moreover let $E_Q(t)\subset \PP n$ be the image by $\ell _O$ of the $t$-dimensional osculating space to $C$ in $Q'$ as defined in Definition \ref{osculating}, i.e. $E_Q(t)=\langle  (t+1)Q'\rangle_C$. Then $$R_X(P)=n+2-t$$  for all $P\in E_Q(t)\backslash E_Q(t-1)$ and each point of $E_Q(2)\setminus \{Q\}$ has $X$-rank $n$.
\end{proposition}

\begin{proof}
The definition of $X\subset \PP n$ implies $p_a(X)=1$ and the existence of a rational normal curve $C \subset \mathbb {P}^{n+1}$ such that for $Q'\in C$ and $O\in T_{Q'}C\backslash \{Q'\}$. Thus $X = \ell _O(C)$. Remark that $\dim (E_Q(t))=t-1$ and $E_Q(1) =\{Q\}$. The line $E_Q(2)$ is the reduction of the tangent cone of $X$ at $Q$.\\
Now fix an integer $t \ge 2$. Since $R_C(O)=n+1$ (see e.g. Theorem 3.8 in \cite{bgi}) and $R_C(P') = n+2-t$ for all $P'\in \langle  (t+1)Q'\rangle \setminus \langle  tQ'\rangle$, the Theorem 3.13 in \cite{bgi} gives that $R_X(P)=n+2-t$ for all $P\in E_Q(t)\backslash E_Q(t-1)$. In particular each point of $E_Q(2)\setminus \{Q\}$ has $X$-rank $n$.
\end{proof}

\providecommand{\bysame}{\leavevmode\hbox to3em{\hrulefill}\thinspace}
\section{$X$-rank of subspaces}\label{subspaces}

In this section we study the $X$-rank of subspaces as we defined it in Definition \ref{rankV} with respect to any integral, non-degenerate projective variety $X\subset \PP n$ and we will get the bound (\ref{ultimo}). Then we will discuss the case of the $X$-rank of lines with respect to a curve $X\subset \PP n$ for $n\geq 4$ in which we can give a precise statement.

\begin{proposition}\label{i0}
Let $X \subset \mathbb {P}^n$ be an integral and non-degenerate $m$-dimensional subvariety. Let $V \subset \mathbb {P}^n$be a projective linear subspace such that $V\cap X=\emptyset$. Then $$R_X(V) \le n+1-m.$$
\end{proposition}

\begin{proof}
Since $V\cap X =\emptyset$, the linear system $\Gamma$ cut out on $X$ by the set of all hyperplanes containing $V$ has no base points. Hence, by Bertini's theorem, if  $H\in \Gamma$ is general, the scheme $X\cap H$ is reduced and of pure dimension $m-1$; if $m \ge 2$ then $X\cap H$ is also an integral scheme by another t. Since $X$ is connected, the exact sequence
\begin{equation}\label{eqi0}0 \to \mathcal {I}_X \to \mathcal {I}_X(1) \to \mathcal {I}_{X\cap H}(1)\to 0
\end{equation}
shows that $X\cap H$ spans $H$.
\\
To get the case $m=1$ it is sufficient to take any $S \subset X\cap H$ with $\sharp (S) = n$ and spanning $H$.
\\
Now we can proceed by induction on $m$ and $n$. Assume first that $m':= 1$ and $n':= n-m+1$ and get the statement when $X$ is a curve. Now assume for the induction procedures for $m\geq 2$  that the statement is true for $(m-1)$-dimensional varieties in $\mathbb{P}^{n-1}$, and use (\ref{eqi0}) to show that the proposition is true also for $\dim(X)=m$ and $X\subset \mathbb{P}^{n}$. Hence we can take $n+1-m$ of the $\deg (X)$ points of $X\cap H$ spanning $H = \mathbb {P}^{n'-1}$ and conclude.
\end{proof}

Now we want to study the $X$-rank of a line $L\subset \PP n$ with respect to an integral and non-degenerate curve $X\subset \mathbb{P}^{n}$. 

Consider the following constructions. Let
\begin{equation}\label{lQ}
\ell_{Q}: \PP n \setminus Q \to \PP {n-1}
\end{equation}
be the linear projection of $\PP n$ onto $\PP {n-1}$ from a point $Q\in X$ and call $C_{Q}\subset \PP {n-1}$ the closure in $\PP {n-1}$ of the integral curve $\ell_{Q}(X\setminus \{Q\})$.

Analogously let
\begin{equation}\label{lL}
\ell_{L}: \PP n \setminus L \to \PP {n-2}
\end{equation}
be the linear projection of $\PP n$ onto $\PP {n-2}$ from a line $L \subset \PP n$ and call $C_{L}\subset \PP {n-2}$ the closure in $\PP {n-2}$ of the integral curve $\ell_{L}(X\setminus (X\cap L))$.

\begin{theorem}\label{i1}
Let $X \subset \mathbb {P}^n$, for $n \geq 4$, be an integral non-degenerate curve of degree $d$ and $L \subset \mathbb {P}^n$ be a line. Let $\ell_Q$ and $\ell_L$ be the linear projections defined in (\ref{lQ}) and (\ref{lL}) respectively such that $b_L:= \deg (\ell _L\vert_{X\backslash \{L\}})$and $b_Q:= \deg (\ell _Q\vert_{X\backslash \{Q\}})$, and let $C_{Q}:=\overline{\ell_{Q}(X)}\subset \PP {n-1}$ and $C_{L}:=\overline{\ell_{L}(X)}\subset \PP {n-2}$ as above. Then
\begin{enumerate}
\item\label{empty} If $L\cap X = \emptyset$, then $R_X(L) \leq n$.
\item\label{notempty} If $L\cap X \neq \emptyset$, then  $\sharp ((X\cap L)_{red})\geq 2$ if and only if $R_X(L)=2$.\\
If $(X\cap L)_{red} =\{Q\}$, then 
$$length(L\cap X) + b_L\cdot \deg (C_L) = d = m_X(Q) + b_Q\cdot \deg (C_Q),$$
where $m_X(Q)$ denotes the multiplicity of $X$ at $Q$. Moreover 
\begin{enumerate}
 \item\label{b1} if $R_X(L) = n$, then $C_L$ is a rational normal curve and $b_Q=b_L$.
\item\label{b2} if $C_L$ is a rational normal curve and $b_L=1$, then $R_X(L) = n+1$.
\end{enumerate}
\end{enumerate}
\end{theorem}

\begin{proof}
 Part (\ref{empty}) is a consequence of  Proposition \ref{i0} applied for $m=\dim (V)=1$. 

Part (\ref{notempty}) for the case $\sharp((X\cap L)_{red})\ge 2$ is obvious. Assume therefore $(X\cap L)_{red} = \{Q\}$. Since we are in characteristic zero, the $X$-rank of a point $O\in \mathbb {P}^n$ is $R_X(O)\le n$ for all $O\in \mathbb {P}^n$ (see \cite{lt}, Proposition 5.1). Hence if $P\in L\backslash \{Q\}$, we get that $L=\langle P,Q\rangle$ and then clearly $R_X(L) \le n+1$.\quad 

Here we prove part (\ref{b1}).
First assume that $C_L$ is not a rational normal curve. Since $n \ge 4$, there is a finite set of points $A\subset C_L$ such that $a:= \sharp (A) \le n$ and $\dim (\langle  A\rangle )=a-2$. Let $A'\subset X\backslash \{Q\}$ such that $\sharp (A')=a$ and $\ell _L(A')=A$. Since the points of $A$ are linearly dependent, the definition of $\ell _L$ implies $L \subseteq \langle  \{Q\}\cup A'\rangle$. Hence $R_X(L)\le a$. 
\\
Now assume that $C_L$ is a rational normal curve and that $b_Q< b_L$. Hence there are $A_1,A_2\in X\backslash \{Q\}$ such that $\ell _L(A_1)=\ell _L(A_2)$ and $\ell _Q(A_1) \ne \ell _Q(A_2)$. Since $Q\notin \{A_1,A_2\}$ and $\ell _Q(A_1) \ne \ell _Q(A_2)$, the subspace $\langle  \{Q,A_1,A_2\}\rangle$ is actually a plane. But $\ell _L(A_1)=\ell _L(A_2)$, hence $P \in \langle  \{Q,A_1,A_2\}\rangle$ for any $P\in L$. Hence $R_X(L) \le 3 <n$. 

Here we prove part (\ref{b2}). Since $R_X(L) \le n+1$, it is sufficient to prove $R_X(L) \ge n+1$. Let $S\subset X$ be a subset of points of $X$ computing $R_X(L)$ (i.e. $S$ is a minimal set of points of $X$ such that $L\subset \langle S\rangle$). 

First assume $Q\in S$. Set $S':=\ell _L(S\backslash \{Q\})$ and take $S''\subseteq S\backslash \{Q\}$ such that $S' = \ell _L(S'')$ and $\sharp (S'') = \sharp (S')$. Since $L \subseteq \langle  S\rangle$, we get that $\dim (\langle  S'\rangle ) = \dim (\langle  S\rangle )-2$. Since $L\cap X = \{Q\}$, the minimality of $S$ gives $S'' = S\backslash \{Q\}$. Now, since $\dim (\langle  S'\rangle ) = \dim (\langle  S\rangle )-2$, while $\sharp (S') = \sharp (S'') =\sharp (S)-1$, the set $\ell _L(S')$ is linearly dependent in $\mathbb {P}^{n-2}$. Since $C_L$ is a rational normal curve of $\mathbb {P}^{n-2}$ and $S' \subset C_L$, the linear dependence of $S'$ implies $\sharp (S') \ge n$. Hence $R_{X}(P) \geq n$ for any $P\in X$ and then $R_{X}(L) \geq n+1$.

Now assume $Q\notin S$. Since $(L\cap X)_{red} = \{Q\}$, we have $L\cap S= \emptyset$. Since $R_X(L) \ge 3$  we get $\sharp (\ell _L(S)) = \sharp (S)$. Since $L \subset \langle S\rangle$, we get $\dim (\langle \ell _L(S)\rangle ) =\dim (\langle S\rangle )-2$. Hence there is $A\subset \ell _L(S)$ such that $\sharp (A) = \sharp (S)-2$ and $\langle A\rangle = \langle \ell _L(S)\rangle$. Since $\ell _L(S)\subset C_L$ and $C_L$ is a rational normal curve of $\PP {n-2}$, we get $\sharp (\ell _L(S)) \ge n+1$.
\end{proof}

\begin{remark}
Take $X, L$ and $Q$ as in part (\ref{b2}) of Theorem \ref{i1}, then the line $L$ is contained in the Zariski tangent space of $X$ at $Q$.
\end{remark}

\begin{proposition}\label{i3} 
Let $X \subset \mathbb {P}^n$ be an integral and non-degenerate $m$-dimensional subvariety and let $V \subset \PP n$ be a linear subspace.  Then 
\begin{equation}\label{ultimo}
R_X(V) \leq n+2-m+\dim (\langle  (X \cap  V)_{red} \rangle).
\end{equation}
\end{proposition}

\begin{proof} 
Let $A\subset (X\cap V)_{red}$ be a finite set of points such that $\sharp (A)=s+1$ and $\langle  A \rangle = \langle  (X\cap V)_{red}\rangle$. Let $N \subset V$ be a complementary subspace of $\langle  (X\cap V)_{red}\rangle$, i.e. a linear subspace such that $N\cap \langle  (X\cap V)_{red}\rangle = \emptyset$ and $\langle  N,(X\cap V)_{red}\rangle=V$. By Proposition \ref{i1} there is a finite subset of points $B\subset X$ such that $\sharp (B) \le n+1-m$ and $N \subseteq \langle B\rangle$. Hence we have $V \subseteq \langle  A\cup B\rangle$. \end{proof}


\begin{thebibliography}{99}

\bibitem{a}  B. \r{A}dlandsvik, Joins and higher secant varieties. Math.Scand. {\bf 62}, 213-222, (1987).
\bibitem{a2} B. \r{A}dlandsvik, Varieties with an extremal number of degenerate higher secant varieties. J. Reine Angew.
Math. {\bf 392}, 16--26, (1988).
\bibitem{accf} L. Albera, P. Chevalier, P. Comon and A. Ferreol, On the virtual array concept for higher order array processing. IEEE Trans. Sig. Proc.,  {\bf 53},1254 1271, (2005).
\bibitem{ah} J. Alexander and A. Hirschowitz, Polynomial interpolation in several variables. J. Algebraic Geom. {\bf 4}, 201-222, (1995).
\bibitem{be} A. Bernardi, Ideals of varieties parameterizing certain symmetric tensors. J. Pure Appl. Algebra, {\bf 212}, 6, 1542-1559,  (2008).
\bibitem{bgi} A. Bernardi, A. Gimigliano and M. Id\`{a}, Computing symmetric rank for symmetric tensors. Preprint: {\tt http://arXiv.org/abs/math/0908.1651v4}, J. Symbolic. Comput. (to appear).
\bibitem{bo} M. C. Brambilla and G. Ottaviani, On the Alexander-Hirschowitz theorem. 
J. Pure Appl. Algebra,  {\bf 212}  (2008) 5, 1229--1251.
\bibitem{bl} J. Buczynski and J. M. Landsberg, Ranks of tensors and a generalization of secant varieties. Preprint: {\tt http://arXiv.org/abs/math/0909.4262}, (2009).
\bibitem{ca} J. F. Cardoso, Blind signal separation: statistical principles. Proc. of the IEEE, 90:2009-2025, October 1998. special issue, R.W. Liu and L. Tong eds. (1998).
\bibitem{c} K. A. Chandler, A brief proof of a maximal rank theorem for generic double points in projective space. Trans. Amer. Math. Soc. {\bf 353}, 1907-1920,  (2001).
\bibitem{clf} H. -T. Chen, T.-L. Liu and C.-S. Fuh,  Learning effective image metrics from few  
pairwise examples. Proceedings of IEEE International Conf. on Computer Vision, pp. 
1371--1378, Beijing, China, Octobor 2005. (2005). 
\bibitem{cs} G. Comas, M. Seiguer, On the rank of a binary form. Preprint: {\tt http://arXiv.org/abs/math/0112311} (2001).
\bibitem{co} P. Comon, Independent Component Analysis. In J-L. Lacoume, editor, Higher Order Statistics,  Elsevier, Amsterdam, London, 29-38, (1992).
\bibitem{cglm} P. Comon, G. Golub, L.-H. Lim and B. Mourrain, Symmetric  tensors and symmetric tensor rank. SIAM Journal on Matrix Analysis  Appl., {\bf 30}, 1254-1279, (2008). 
\bibitem{cst} P. Comon, M. S\o rensen,  E. P. Tsigaridas, In 35th Int. Conf. on Acoustics, Speech, and Signal Processing (ICASSP), Dallas, USA, March (2010). (To appear)
\bibitem{cr} P. Comon and M. Rajih, Blind identification of under-determined mixtures based on the characteristic function. Signal Processing, {\bf 86}, 2271- 2281, (2006).
\bibitem{dh} D. L. Donoho and X. Huo, Uncertainty principles and ideal atomic decompositions. IEEE Trans. Inform. Theory, {\bf 47}, 2845-2862, (2001).
\bibitem{f} B. Fantechi, On the superaddivity of secant defects. Bull. Soc. Math. France {\bf 118} (1990), no. 1, 85--100.
\bibitem{ga} H. Gang, Face Recognition by Discriminative Orthogonal Rank-one Tensor Decomposition, invited Book Chapter, in Recent Advances in Face Recognition, Edited by Marian Stewart Bartlett, Kresimir  Delac and and  Mislav Grgic, (2008).  (ISBN: 978-953-7619-34-3)
\bibitem{js} T. Jiang and N.  Sidisopoulos,  Kruskal's permutation lemma and the identification of CANDECOMP/PARAFAC and bilinear models. IEEE Trans. Sig. Proc., {\bf 52}, 2625-2636, (2004).
\bibitem{lt} J. M. Landsberg and Z. Teitler, On the ranks and border ranks of symmetric tensors. Found. Comput. Math. (2010) \textbf{10}: 339--366. DOI 10.1007/510208-009-9055-3.
\bibitem{p} R. Piene, Cuspidal projections of space curves. Math. Ann. {\bf 256}, 95-119, (1981).
\bibitem{Sy} J. Sylvester, Sur une extension d'un th\'eor\`eme de Clebsh relatif aux courbes du quatri\`eme degr\'e. Comptes Rendus, Math. Acad. Sci. Paris, {\bf 102}, 1532-1534, (1886).
\bibitem{tp} M. Turk, A. Pentland. Eigenfaces for recognition. Journal of Cognitive Neuroscience {\bf 3} (1): 71--86  (1991).
\bibitem{z} F. L. Zak, Tangents and secants of algebraic varieties,
Translated from the Russian manuscript by the author. Translations of Mathematical Monographs, {\bf 127}. American Mathematical Society, Providence, RI, (1993).
\end{thebibliography}
\end{document}